\documentclass[a4paper,11pt,hypertex,hyperref,twoside]{amsart}
\usepackage{amsthm}
\usepackage{amsmath}
\usepackage{amssymb}
\usepackage{amsfonts}
\usepackage{amscd}
\usepackage{float}
\usepackage{color}
\usepackage[english]{babel}
\usepackage{nameref}
\usepackage{hyperref}
\usepackage{textcomp}
\usepackage{enumerate}

\newtheorem{thm}{Theorem}
\newtheorem{cor}{Corollary}
\newtheorem{prop}{Proposition}
\newtheorem*{rem}{Remark}

\newtheorem*{defi}{Definition}
\newtheorem{lem}{Lemma}
\usepackage{fullpage}

\begin{document}

    \title{Dynamical systems of simplices in dimension $2$ or $3$}
    \author{GERALD BOURGEOIS, Sebastien ORANGE}
    \date{\today}
    
    \address{Bourgeois Gerald, French Polynesia University, BP 6570, 98702 FAA'A, Tahiti, French Polynesia.}
    \email{bourgeois.gerald@gmail.com}
    \address{S\'ebastien Orange, Le Havre University, 25 rue Philippe Lebon, 76600 Le Havre, FRANCE.}
    \email{Sebastien.Orange@lip6.fr}
    

\begin{abstract}
Let $\mathcal{T}_0=(A_0^0\cdots{A_0^d})$ be a $d$-simplex, $G_0$ its centroid, $\mathcal{S}$ its circumsphere, $O$ the center of $\mathcal{S}$. Let $(A_1^i)$ be the points where $\mathcal{S}$ intersects the lines $(G_0A_0^i)$, $\mathcal{T}_1$ the $d$-simplex $(A_1^0\cdots{A_1^d})$, and $G_1$ its centroid. By iterating this construction, a dynamical system of $d$-simplices $(\mathcal{T}_i)$ with centroids $(G_i)$ is constructed. For $d=2$ or $3$, we prove that the sequence $(OG_i)_i$ is decreasing and tends to $0$. We consider the sequences $(\mathcal{T}_{2i})_i$ and  $(\mathcal{T}_{2i+1})_i$; for $d=2$ they converge to two equilateral triangles with at least quadratic speed; for $d=3$ they converge to two isosceles tetrahedra with at least geometric speed. In this last case, we give an explicit expression of the lengths of the edges of the limit form. We show also that if $\mathcal{T}_0$ is a planar cyclic quadrilateral then $(\mathcal{T}_n)$ converges to a rectangle with at least geometric speed or eventually to a square with a speed that is conjectured as cubic. The proofs are largely algebraic and use Gr\"obner basis computations.
\end{abstract}

\maketitle
    \section{Introduction}
    \subsection{The general problem }
    Let $\mathcal{T}_0=(A_0^0\cdots{A_0^d})$ be a $d$-simplex, $G_0$ its centroid, $\mathcal{S}$ its circumsphere in $\mathbb{R}^d$, $O$ the center of $\mathcal{S}$. Let $(A_1^i)$ be the points where $\mathcal{S}$ intersects the lines $(G_0A_0^i)$ respectively, $\mathcal{T}_1$ the $d$-simplex $(A_1^0\cdots{A_1^d})$, and $G_1$ its centroid. If we iterate this construction then we produce a dynamical system of $d$-simplices $(\mathcal{T}_i)$ with centroids $(G_i)$.
    
    Let $M,N$ be two points of $\mathbb{R}^d$; $MN$ refers to the euclidean norm of the vector $\overrightarrow{MN}$.

Numerical investigations, using Maple, with thousands of random
simplices in dimensions up to 20 indicate that:
\begin{enumerate}[i)]
 \item the sequence $(OG_i)_i$ is decreasing and tends to $0$ (from~\cite{10}).
 \item the sequences $(\mathcal{T}_{2i})_i$ and  $(\mathcal{T}_{2i+1})_i$ converge to two $d$-simplices with centroid $O$.
\end{enumerate}

\begin {rem}
The condition ``circumcenter=centroid'' for a $d$-simplex is equivalent
to the statement that the sum of the squares of the edge lengths of
each facet is the same for all $d+1$ facets.
\end {rem}
We prove the assertions i) and ii) in the cases $d=2$ and $d=3$. We conjecture that the result is valid for any $d$.\\
From now on, we assume that the radius of $\mathcal{S}$ is $1$.    
\subsection{The case $d=3$}
$\mathcal{T}_i=(A_iB_iC_iD_i)$ is a tetrahedron such that its four vertices are not coplanar. These vertices are ordered and $\mathcal{T}_i$ is viewed as a point of $\mathcal{S}^4$\\
Let $\phi$ and $\psi$  be the functions which transform $\mathcal{T}_0$ into $\mathcal{T}_1$ and $G_1$. 
\begin{rem} \begin{enumerate}[i)]
             \item $\phi$ does not admit any non planar tetrahedron as fixed point.
             \item If $\mathcal{T}$ is an isosceles tetrahedron then $\mathcal{T}$ is a fixed point of $\phi\circ\phi$.
            \end{enumerate}
\end{rem}

Our main result is the following:

\indent The sequences $(\mathcal{T}_{2i})_{i\in\mathbb{N}}$ and $(\mathcal{T}_{2i+1})_{i\in\mathbb{N}}$ are well defined and converge, with at least geometric speed, to two non planar isosceles tetrahedra that are symmetric with respect to $O$.\\
\indent We consider also the degenerate case where $\mathcal{T}_0$ is a planar cyclic convex quadrilateral; then the sequence $(\mathcal{T}_{i})_i$ converges, with at least geometric speed, to a rectangle. If $\mathcal{T}_0$ is a harmonic quadrilateral then the limit form is a square and we conjecture that the convergence is with order three. \\
\indent Moreover, in both cases, we give explicit expressions for the lengths of the edges of the limit form from the ones of $\mathcal{T}_0$.

\subsection{The case $d=2$}
$\mathcal{T}_i=(A_iB_iC_i)$ is a triangle such that its vertices are pairwise distinct.

We prove that the sequences $(\mathcal{T}_{2i})_{i\in\mathbb{N}}$ and $(\mathcal{T}_{2i+1})_{i\in\mathbb{N}}$ are well defined and converge to two equilateral triangles that are symmetric with respect to $O$.
   Moreover, these sequences converge with at least quadratic speed.\\
 \indent Thus if the dimension is two, it can be observed a much more quick convergence that in the dimension three case.
 
\subsection{Method used}

Along this paper, we use rational functions of the square of the lengths of the edges of the $d$-simplices and naturally some systems of polynomial equations appear. Computations with such systems requires Gr\"obner basis softwares. For this, we have used the J. C. Faugere's software ``FGb'' (see~\cite{8ter}) and the computer algebra system \textsc{Maple} (see~\cite{Maple}) which provides some other tools we needed.

 \section {Standard definitions and results about tetrahedra.}
 
 \subsection{General tetrahedra}  
 
 Let $\mathcal{T}=(ABCD)$ be a tetrahedron, $G$ its centroid and $\mathcal{V}(\mathcal{T})$ its volume. We assume $A,B,C,D$ are not coplanar and the circumsphere of $\mathcal{T}$ has center $O$ and radius $1$. Let $BC=a$, $CA=b$, $AB=c$, $AD=a'$, $BD=b'$, $CD=c'$.

\begin{center}
\textit{Remark that the 3-tuple $(a,b,c)$ does not play the same role as $(a',b',c')$.}\end{center}

\begin{prop}\label{first_equality}  For all $\mathcal{T}$, we have
\begin{enumerate}[i)]
 \item $O{G^2}=1-\dfrac{1}{16}(a^2+b^2+c^2+a'^2+b'^2+c'^2)\\\phantom{O{G^2}}=1-\dfrac{1}{4}(GA^2+GB^2+GC^2+GD^2)$. (See~\cite{2} p. 64 and ii)).
 \item  $GA^2=\dfrac{3}{16}(a'^2+b^2+c^2)-\dfrac{1}{16}(a^2+b'^2+c'^2),\\ GB^2=\dfrac{3}{16}(b'^2+c^2+a^2)-\dfrac{1}{16}(b^2+c'^2+a'^2),\\ GC^2=\dfrac{3}{16}(c'^2+a^2+b^2)-\dfrac{1}{16}(c^2+a'^2+b'^2),\\GD^2=\dfrac{3}{16}(a'^2+b'^2+c'^2)-\dfrac{1}{16}(a^2+b^2+c^2)$. (see~\cite{4} p. xv).
\end{enumerate}
\end{prop}

If $u,v,w\in\mathbb{R}^3$ then $Gram(u,v,w)$ refers to the Gram matrix of these three vectors. If $U$ is a square matrix then $det(U)$ refers to its determinant.
\begin{prop}\label{second_equality}
Let $\Gamma(A,B,C,D)=det\begin{pmatrix}0&1&1&1&1\\1&0&c^2&b^2&a'^2\\1&c^2&0&a^2&b'^2\\1&b^2&a^2&0&c'^2\\1&a'^2&b'^2&c'^2&0\end{pmatrix}$ be the Cayley-Menger determinant, and let $\Delta(A,B,C,D)=det\begin{pmatrix}0&c^2&b^2&a'^2\\c^2&0&a^2&b'^2\\b^2&a^2&0&c'^2\\a'^2&b'^2&c'^2&0\end{pmatrix}$.

With these notations, we have
\begin{enumerate}[i)]
 \item The volume as a function of $a,b,c,a',b',c'$ (see~\cite{6} p. 168):\\
$288\times\mathcal{V}(\mathcal{T})^2=8\times{det}({Gram(\stackrel{\rightarrow}{AB},\stackrel{\rightarrow}{AC},\stackrel{\rightarrow}{AD})})=\Gamma(A,B,C,D)$.\\
 Therefore, $A,B,C,D$ are not coplanar if and only if $\Gamma(A,B,C,D)>0$.
 \item Let $R$ be the circumradius of the tetrahedron (here $R=1$).
 \begin{enumerate}[a)]
  \item  Crelle's formula (1821): $6R\,\mathcal{V}(\mathcal{T})=\delta$ where $\delta$ denotes the area of the triangle with sides of lengths $aa',bb',cc'$  (see~\cite{2} p. 250) 
  \item $R^2=-\dfrac{\Delta(A,B,C,D)}{2\Gamma(A,B,C,D)}$ (see~\cite{6} p. 168).
 \end{enumerate}
 \item There exists a relation between $a,b,c,a',b',c'$:  $det\begin{pmatrix}1/2&1&1&1&1\\1&0&c^2&b^2&a'^2\\1&c^2&0&a^2&b'^2\\1&b^2&a^2&0&c'^2\\1&a'^2&b'^2&c'^2&0\end{pmatrix}=0.$
\end{enumerate}

\end{prop}
\begin {proof}
iii) comes from ii)b) and $R=1$: $\Gamma(A,B,C,D)+\dfrac{1}{2}\Delta(A,B,C,D)=0$.
\end{proof}
\begin{prop}\label{algebraic_quantities}
 We assume $\mathcal{T}$ is a non planar tetrahedron or a cyclic quadrilateral.\\
  The expression $Pt(\mathcal{T})=-\Delta(A,B,C,D)$ has the following properties:
 \begin{enumerate}[i)]
  \item $Pt(\mathcal{T})=(bb'+cc'-aa')(cc'+aa'-bb')(aa'+bb'-cc')(aa'+bb'+cc')\\ \phantom{Pt(\mathcal{T})}=
2a^2a'^2b^2b'^2+2b^2b'^2c^2c'^2+2c^2c'^2a^2a'^2-a^4a'^4-b^4b'^4-c^4c'^4$.
\item If $\mathcal{T}$ is not planar then each factor of the product appearing in i) is positive (Ptolemy's inequality, see~\cite{7}). (see also~\cite {5} p.549,555).\\
 Moreover $Pt(\mathcal{T})=576\times{\mathcal{V}(\mathcal{T})}^2\times{R^2}$ (From Proposition~\ref{second_equality}). 
\item  If $\mathcal{T}$ is a cyclic quadrilateral then $Pt(\mathcal{T})=0$. 
 \end{enumerate}
 \end{prop}

Conversely let $a,b,c,a',b',c'$ be six positive reals. 
 Does there exist a tetrahedron such that the lengths of its edges are these reals?

\begin{prop}\label{nsc_tetrahedron}
The necessary and sufficient condition that the six lengths form a tetrahedron (eventually planar) appears to be:
 \begin{enumerate}[i)]
  \item  $a,b,c$ are the lengths of a triangle that is $16\times{S^2}=2a^2b^2+2b^2c^2+2c^2a^2-a^4-b^4-c^4\geq{0}$ (then $S$ is the area of this triangle).\
\item $det\begin{pmatrix}0&1&1&1&1\\1&0&c^2&b^2&a'^2\\1&c^2&0&a^2&b'^2\\1&b^2&a^2&0&c'^2\\1&a'^2&b'^2&c'^2&0\end{pmatrix}\geq{0}$ 
 \end{enumerate}

\end{prop}

\begin{proof} (See also~\cite{5} p. 545)
Assertion $i)$ of Proposition~\ref{algebraic_quantities} implies that the condition $ii)$ is necessary.\\
The following can be generalized in any dimension:\\
The tetrahedron exists if and only if ${Gram(\stackrel{\rightarrow}{AB},\stackrel{\rightarrow}{AC},\stackrel{\rightarrow}{AD})}$ is a symmetric non negative matrix. That gives the conditions $i),ii)$.  
\end{proof}
\subsection{Isosceles (or equifacetal) tetrahedra.} {\ }
\begin{defi} Let $\mathcal{T}$ be a non planar tetrahedron whose circumradius is $1$. The tetrahedron $\mathcal{T}$ is said to be isosceles (or equifacetal) if $a=a',b=b',c=c'$.
\end{defi}

\begin{prop} \label{isosceles}We have
\begin{enumerate}[i)]
 \item $\mathcal{T}$ is isosceles if and only if \;$G=O$. (see~\cite{1} p. 197).
 \item If  $\mathcal{T}$ is isosceles then  \;$72\;\mathcal{V}^2(\mathcal{T})=(b^2+c^2-a^2)(c^2+a^2-b^2)(a^2+b^2-c^2)$ and $a^2<b^2+c^2$, $b^2<c^2+a^2$, $c^2<a^2+b^2$. (see~\cite{3} p. 205).
 \item For all $\mathcal{T}$, $0<aa'+bb'+cc'\leq{8}$. 
Moreover $aa'+bb'+cc'=8$ if and only if $\mathcal{T}$ is isosceles. (see~\cite{5} p. 558).
\end{enumerate}
\end{prop}

The following result is a direct consequence of propositions~\ref{isosceles} and~\ref{first_equality}.

\begin{prop}\label{medianes_tetraedre}
Let $\mathcal{T}=(ABCD)$ be a tetrahedron eventually planar whose centroid is $G$.\\  
$GA=GB=GC=GD$ if and only if $\mathcal{T}$is an isosceles tetrahedron or a planar rectangle.
\end{prop}

    \section{Deformation from $\mathcal{T}_0$ to $\mathcal{T}_1$ ($d=3$)}
\subsection{Parameters and notations}
We adapt our preceding notations; $A_0$, $B_0$, $C_0$, $D_0$ are four points of $\mathcal{S}$. We assume only $A_0$, $B_0$, $C_0$, $D_0$ are not all equal. Then the functions $\phi,\psi$ are continuous in $\mathcal{T}_0$.
\begin{rem}
i) The following calculations are valid if $A_0$, $B_0$, $C_0$, $D_0$ are four points, not all equal, of $\mathcal{C}$, a planar circle of center $O$ and radius $1$.\\
ii) Let $\Delta=\{(P,P,P,P): P\in\mathcal{S}\}$. $\phi$ is defined on $\mathcal{S}^4\backslash\Delta$.
\end{rem}

\begin{defi}
The parameters $(d_{ij})_{i<j}$ of $\mathcal{T}_0$ are the square of the lengths of their edges:
$d_{12}=A_0{B_0}^2,d_{13}=A_0{C_0}^2,d_{14}=A_0{D_0}^2,d_{23}=B_0{C_0}^2,d_{24}=B_0{D_0}^2,d_{34}=C_0{D_0}^2$. Let us recall that they are linked by a polynomial equality (see~Assertion $iii)$ of Proposition~\ref{second_equality}) and two polynomial inequalities (see~Proposition~\ref{nsc_tetrahedron}).
\end{defi}
Let $g_1=G_0{A_0}^2,g_2=G_0{B_0}^2,g_3=G_0{C_0}^2,g_4=G_0{D_0}^2$ and let $p_0=1-O{G_0}^2=\dfrac{1}{4}(g_1+g_2+g_3+g_4)$ and $p_1$ be the opposites of the powers of $G_0$ and $G_1$ with respect to $\mathcal{S}$. According to Proposition $1$, $(g_i)_i$ and $p_0$ are polynomial functions of $(d_{ij})_{ij}$.
\begin{prop}\label{equations}We have
\begin{enumerate}[i)]
 \item The parameters of $\mathcal{T}_1$ are $\left( d_{ij}^1={p_0}^2\dfrac{d_{ij}}{g_ig_j}\right)_{i<j}$.
 \item $p_1=\dfrac{{p_0}^2}{16}\sum_{i<j}\dfrac{d_{ij}}{g_ig_j}$.
 \item $O{G_0}^2$ and $O{G_1}^2$ are rational functions of the $(d_{ij})_{ij}$.
 \item $\mathcal{V}(\mathcal{T}_0)^2$ and $\mathcal{V}(\mathcal{T}_1)^2$ are rational functions of the $(d_{ij})_{ij}$.
\end{enumerate}
\end{prop}
\begin{proof}
For $i)$: $G_0A_0\times G_0A_1=p_0$ and ${G_0A_1}^2=\dfrac{{p_0}^2}{g_1}$. The triangles $(G_0A_0B_0)$ and $(G_0B_1A_1)$ are similar; then $\dfrac{A_1B_1}{A_0B_0}=\dfrac{G_0A_1}{G_0B_0}$ and ${A_1B_1}^2=d_{12}\dfrac{{p_0}^2}{g_1g_2}$. 

We deduce $ii)$ from $i)$ and $iii)$ from $ii)$. We deduce $iv)$ from $i)$ and Proposition~\ref{second_equality}. $i)$.
\end{proof}

\subsection{Inequalities.}{\ }

Let $\mathcal{T}$ be a tetrahedron or a cyclic quadrilateral.

Let $\mathcal{P}(\mathcal{T})$ be the property: ``the edges of $\mathcal{T}$ satisfy $d_{12}=d_{34},d_{13}=d_{24},d_{14}=d_{23}$'' i.e. ``$\mathcal{T}$ is an  isosceles tetrahedron or a planar rectangle''.

The first key is the following result:
 \begin{thm}\label{inequalities} We have the following inequalities:
\begin{enumerate}[i)]
 \item $O{G_1}\leq{O{G_0}}$;
 \item for any $(ijkl)\in\{(1234),(1324),(1423)\}$, $d_{ij}d_{kl}\leq{d_{ij}^1d_{kl}^1}$;
 \item $Pt(\mathcal{T}_0)\leq{Pt(\mathcal{T}_1)}$ (see~Proposition 3); 
\end{enumerate}
Moreover, for \textit{$i)$},\textit{ $ii)$} and \textit{$iii)$}, equalities stands if and only if $\mathcal{P}(\mathcal{T}_0)$ holds.
\end{thm}

\begin{proof}
$\bullet$ For $i)$: 
 the following choice of unknowns enables us to conclude; we may assume $g_1-g_2=s_1,g_2-g_3=s_2,g_3-g_4=s_3,g_4s_4=1$ where $s_1,s_2,s_3\geq{0}$ and $s_4>0$.
Let's set $E=64\dfrac{(p_1-p_0)g_1g_2g_3}{p_0}$.

We note that $signum(E)=signum(O{G_0}^2-O{G_1}^2)$ and we obtain, using the J. C. Faugere's software ``FGb'', this miraculous result:\\
$E=d_{23}(s_1-s_3)^2+16 {s_2}^2 {s_3}^2 s_4+4 d_{34} s_1 s_2+12 s_1 {s_2}^2+4 s_1 s_2 s_3+20 s_2 {s_3}^2+21 s_2 {s_3}^3 s_4+d_{24} {s_3}^2+d_{34} {s_1}^2+4 s_1 {s_3}^2+d_{24} {s_1}^2 s_3 s_4+d_{34} {s_1}^2 s_3 s_4+12 {s_3}^3+2 d_{24} s_1 s_2 s_3 s_4+8 {s_2}^3+3 d_{24} s_1 {s_3}^2 s_4+6 d_{34} s_1 s_2 s_3 s_4+4 d_{34}{ s_2}^2+d_{34} {s_3}^2+5 d_{34} {s_2}^2 s_3 s_4+d_{24} {s_1}^2+3 d_{34} s_1 {s_2}^2 s_4+4 {s_2}^3 s_3 s_4+3 d_{34} s_2 {s_3}^2 s_4+2 d_{24} s_1 s_3+12 {s_2}^2 s_3+2 d_{34} s_1 s_3+4 d_{34} s_2 s_3+9 {s_3}^4 s_4+2 d_{34} {s_2}^3 s_4+d_{34} {s_1}^2 s_2 s_4+{s_1}^2 s_2 s_3 s_4+4 s_1 {s_2}^2 s_3 s_4+{s_1}^2 {s_3}^2 s_4+10 s_1 s_2 {s_3}^2 s_4+6 s_1 {s_3}^3 s_4+3 d_{34} s_1 {s_3}^2 s_4$; obviously $E$ is a non negative real.

Moreover if $E=0$ then $s_1=s_2=s_3=0$ and, according to Proposition~\ref{medianes_tetraedre}, $\mathcal{P}(\mathcal{T}_0)$ holds.\\
$\bullet$ For $ii)$: let $\Lambda=\dfrac{p_0^4}{g_1g_2g_3g_4}$. By Proposition~\ref{equations}, $d_{ij}^1d_{kl}^1=\Lambda{d_{ij}d_{kl}}$. According to the AM-GM inequality, $\Lambda\geq{1}$ and $d_{ij}d_{kl}\leq{d_{ij}^1d_{kl}^1}$.

Moreover if $d_{ij}d_{kl}={d_{ij}^1d_{kl}^1}$ then $\Lambda=1$ and according to the properties of the AM-GM inequality, $g_1=g_2=g_3=g_4$ and $\mathcal{P}(\mathcal{T}_0)$ holds.\\
$\bullet$ For $iii)$: An easy computation gives $Pt(\mathcal{T}_1)=\Lambda^2{Pt(\mathcal{T}_0)}$ and we reason as for $ii)$. 
\end{proof}

    \section{Solution of the case $d=3$. Part 1}
The $(d_{ij}^n)_{i<j}$ and $(d_{ij})_{i<j}$ refer to the parameters of $\mathcal{T}_n$ and $\mathcal{T}_0$.
\begin{thm}\label{convergence_ogn}
\begin{enumerate}[i)]
\item The sequence $(O{G_n})_n$ tends to $0$.
\item  Let $\mathcal{T}=(ABCD)$ be a cluster point of the bounded sequence $(\mathcal{T}_i)$.
 Then $\mathcal{T}$ is not flat and is isometric to a fixed isosceles tetrahedron that admits the following parameters:\\
  ${d_{12}^\infty}^2=Ld_{12}d_{34},{d_{13}^\infty}^2=Ld_{13}d_{24},{d_{14}^\infty}^2=Ld_{14}d_{23}$ where\\ $L=\dfrac{64}{(\sqrt{d_{12}d_{34}}+\sqrt{d_{13}d_{24}}+\sqrt{d_{14}d_{23}})^2}$.
\end{enumerate}
  \end{thm}

     \begin{proof}
\indent According to Theorem~\ref{inequalities}$\;iii)$, the sequence $(Pt(\mathcal{T}_i))$ is increasing, then is positive and, for all $i$, $A_i,B_i,C_i,D_i$ are not coplanar; for all $i$, $G_i\notin\{A_i,B_i,C_i,D_i\}$ then the sequence $(\mathcal{T}_{i})$ is well defined and the functions $\phi,\psi$ are continuous in $(\mathcal{T}_i)$. Moreover the bounded sequence $(Pt(\mathcal{T}_i))$ converges to $Pt>0$.\\  
 \indent According to Theorem~\ref{inequalities}$\;i)$, the bounded sequence $(O{G_n})$ is decreasing and converges to $r\geq{0}$. We can extract a subsequence $(\mathcal{T}_{n_k})_k$ such that $(A_{n_k})_k$, $(B_{n_k})_k$, $(C_{n_k})_k$, $(D_{n_k})_k$ converge to $A,B,C,D$. Let $\mathcal{T}$ be the tetrahedron $(ABCD)$ and $G$ its centroid. Therefore $(G_{n_k})_k$ converges to $G$, $O{G}=r$ and $Pt(\mathcal{T})=Pt$; thus $A,B,C,D$ are not coplanar and $\phi(\mathcal{T})=\mathcal{T}'=(A'B'C'D')$ is well defined; let $G'$ be its centroid.
 
 Assume $\mathcal{T}$ is not isosceles; then Theorem~\ref{inequalities} $i)$ implies that $OG'<OG$. Let $\epsilon\in ( 0,O{G}-O{G'} )$ and let $\alpha>0$ such that if $||\mathcal{T}-\mathcal{T}_n||<\alpha$ then $O{G_{n+1}}-O{G'}<\epsilon$. There exists an integer $n_k$ such that $||\mathcal{T}-\mathcal{T}_{n_k}||<\alpha$; then $O{G_{n_k+1}}-O{G'}<\epsilon$. Thus $O{G_{n_k+1}}<O{G}=r$, a contradiction. Therefore $\mathcal{T}$ is isosceles and $(O{G_n})_n$ converges to $0$. A parameter $d_{ij}^\infty$ of $\mathcal{T}$ is the limit of the sequence $(d_{ij}^{n_k})_k$.
 
 According to the proof of Theorem~\ref{inequalities} $ii)$, if $\{i,j,k,l\}$ is a permutation of $\{1,2,3,4\}$, then the bounded sequence $(d_{ij}^nd_{kl}^n)_n$ satisfy $d_{ij}^{n+1}d_{kl}^{n+1}=\Lambda_n{d_{ij}^nd_{kl}^n}$; it is increasing then is convergent. Thus the infinite product $\Pi_{n=0}^\infty{\Lambda_n}$ converges to $L\geq{1}$ such that ${d_{ij}^\infty}^2=Ld_{ij}d_{kl}$.\\  $d_{12}^\infty+d_{13}^\infty+d_{14}^\infty=8$ gives the explicit value of $L$ as a function in the $(d_{ij})$.\\  
 Thus there exists an unique value of $d_{ij}^\infty$ that is valid for all cluster points.
 \end{proof}
  \begin{cor}\label{convergence_dij}
 If $i<j$ then the sequence $(d_{ij}^n)_n$ converges to $d_{ij}^\infty$.
 \end{cor}
\begin{proof}
Let $\mathcal{U}=\{(1234),(1324),(1423)\}$. If $(ijkl)\in\mathcal{U}$ then 
$d_{ij}^nd_{kl}^n$ converges to ${d_{ij}^\infty}^2$ and $\sum_{(ijkl)\in\mathcal{U}}\sqrt{d_{ij}^nd_{kl}^n}$ converges to $8$.\\
$2\sum_{(ijkl)\in\mathcal{U}}\sqrt{d_{ij}^nd_{kl}^n}\leq{\sum_{i<j}d_{ij}^n}\leq{16}$; thus $\sum_{(ijkl)\in\mathcal{U}}\left(\sqrt{d_{ij}^n}-\sqrt{d_{kl}^n}\right)^2$ converges to $0$ and if $(ijkl)\in\mathcal{U}$  then $d_{ij}^n-d_{kl}^n$ converges to $0$; therefore $d_{ij}^n$ and $d_{kl}^n$ converge to $d_{ij}^\infty$.
\end{proof}
\indent It remains to prove that the $(\mathcal{T}_{2i})$ cannot turn around $O$. 
    \section{Solution of the case $d=3$. Part 2}
    We assume $\mathcal{T}_0$ is not isosceles. We study the convergence speed of the sequence ${(O{G_n}}^2)$.
    \begin{defi} Let $f_n,g_n$ be two positive sequences.
    \begin{enumerate}
     \item We say that $f_n=\Theta(g_n)$ if and only if there exist two positive reals $\alpha,\beta$ such that, for all sufficiently large $n$, $\alpha{g_n}\leq{f_n}\leq{\beta{g_n}}$.
     \item We say that $f_n\sim{g_n}$ if and only if  $\lim_{n\rightarrow\infty}\dfrac{f_n}{g_n}=1$.
    \end{enumerate}
    \end{defi}
   \subsection{Taylor series I} 
 Let $d_{ij}^n=d_{ij}^\infty+h_{ij}^n$, $h_n=(h_{ij}^n)_{i<j}$ (the sequence $(h_n)$ converges to the zero vector) and $\delta_n=\sum_{i<j}h_{ij}^n$. Let $\epsilon_n=(h_{12}^n-h_{34}^n)^2d_{12}^\infty+(h_{13}^n-h_{24}^n)^2d_{13}^\infty+(h_{14}^n-h_{23}^n)^2d_{14}^\infty$.
 \begin{prop}
 ${OG_{n+1}}^2=\dfrac{-\delta_n}{16}-\dfrac{1}{16^2}\epsilon_n+O(||h_n||^3)$.
 \end{prop}
 \begin{proof}
 Let us recall that $O{G_n}^2=\dfrac{-\delta_n}{16}$. With the help of the software ``FGb'' we put $O{G_{n+1}}^2$ in the form of a rational function in the unknowns $(h_{ij}^n)$ and we deduce the terms of degree at most two of its Taylor series: \\ $O{G_{n+1}}^2=\dfrac{-(\delta_n+\tau_n)}{16\left(1+\dfrac{\delta_n}{4}\right)}+O\left(||h_n||^3\right)=
 \dfrac{-\delta_n}{16}-\dfrac{1}{16}\left(\tau_n-\dfrac{{\delta_n}^2}{4}\right)+O(||h_n||^3)$ where:\\
$\tau_n-\dfrac{{\delta_n}^2}{4}=\dfrac{1}{16}\left((h_{12}^n-h_{34}^n)^2d_{12}^\infty+(h_{13}^n-h_{24}^n)^2d_{13}^\infty+(h_{14}^n-h_{23}^n)^2d_{14}^\infty\right)+O(||h_n||^3)$.
\end{proof}
 
    \subsection{Taylor series II}    
 Let us recall that $d_{14}^\infty=8-d_{12}^\infty-d_{13}^\infty$ and $d_{12}^\infty+d_{13}^\infty>4$.\\
 \indent Now we prove the second key.
    \begin{prop}\label{comp_epsilon_delta}
 $i)$ $\delta_n=O(||h_n||^2)$.\\
 $ii)$ There exists $k>0$ such that $\epsilon_n\geq{-k}\delta_n$.
    \end{prop}
    \begin{proof}
 For i): From assertion $iii)$ of proposition~\ref{second_equality}, $h_n$ satisfies an algebraic relation. We obtain with Maple the Taylor series of the preceding relation; we consider the terms of degree at most two: 
   $4(4-d_{12}^\infty)(4-d_{13}^\infty)(d_{12}^\infty+d_{13}^\infty-4)\delta_n+\sigma'_n+O(||h_n||^3)=0$ where $\sigma'_n=O(||h_n||^2)$ is a non negative quadratic form in the $(h_{ij})$. Thus $\delta_n=O(||h_n||^2)$.\\
\indent For ii): If we take $h_{12}^n=-h_{13}^n-h_{14}^n-h_{23}^n-h_{24}^n-h_{34}^n$, then we transform the expression $\sigma'_n$ in $\sigma_n$ that satisfies :
   $\sigma_n=-4(4-d_{12}^\infty)(4-d_{13}^\infty)(d_{12}^\infty+d_{13}^\infty-4)\delta_n+O(||h_n||^3)$. $\sigma_n$ is a quadratic form in $h'_n=(h_{13}^n,h_{14}^n,h_{23}^n,h_{24}^n,h_{34}^n)\in\mathbb{R}^5$ whose symmetric associated matrix $\Sigma$ is defined by:\\      
\begin{small}\noindent   $\Sigma=\begin{pmatrix}
       \Sigma_{11}&\Sigma_{12}&\Sigma_{12}&\Sigma_{14}&2\Sigma_{12}\\
       *&\Sigma_{22}&2\Sigma_{12} - \Sigma_{22}&\Sigma_{12}&2\Sigma_{12}\\
       *&*&\Sigma_{22}&\Sigma_{12}&2\Sigma_{12}\\
       *&*&*&\Sigma_{11}&2\Sigma_{12}\\
       *&*&*&*&4\Sigma_{12}\\
      \end{pmatrix}$\\      
       with   $\left \lbrace \begin{array}{l}
      \Sigma_{11}=2d_{12}^\infty{d_{13}^\infty}+32-12d_{13}^\infty+{d_{13}^\infty}^2-12d_{12}^\infty+{d_{12}^\infty}^2,\\
   \Sigma_{12}=-4d_{12}^\infty+16+d_{12}^\infty{d_{13}^\infty}-8d_{13}^\infty+{d_{13}^\infty}^2,\\
   \Sigma_{14}=-{d_{12}^\infty}^2+4d_{12}^\infty+{d_{13}^\infty}^2-4d_{13}^\infty,\\ 
  \Sigma_{22}=-4d_{13}^\infty+{d_{13}^\infty}^2.\\
    \end{array}
 \right.$\end{small}
  
By the same way we transform the expression $\epsilon_n$ in $\epsilon'_n$ such that \begin{small}$\epsilon_n=\epsilon'_n+O(||h_n||^3)$.
   $\epsilon'_n$ is a quadratic form in $h'_n$ whose matrix is:\\ 
   $E=\begin{pmatrix}
   E_{11}& E_{12}& E_{12}& E_{14}& 2E_{12}&\\
   *& E_{22}&2 E_{12}-E_{22}& E_{12}&2E_{12}&\\
   *& *& E_{22}& E_{12}&2E_{12}&\\
   *& *&*& E_{11}& 2E_{12}&\\
   *& *& *& *&4E_{12}&\\
   \end{pmatrix}$\\ 
     with   $\left \lbrace \begin{array}{l}
                          E_{11}=d_{12}^\infty+d_{13}^\infty \\
                          E_{12}=d_{12}^\infty \\
                          E_{14}= d_{12}^\infty-d_{13}^\infty\\
                          E_{22}= 8-d_{13}^\infty\\
                         \end{array}
 \right.$         \end{small}
  
   Let $F$ be the plane of $\mathbb{R}^5$ defined by the relations: $\{h_{13}^n=h_{24}^n,h_{14}^n=h_{23}^n,h_{34}^n=-h_{13}^n-h_{14}^n\}$. $F$ is the nullspace of the matrices $\Sigma$ and $E$.
 \begin{defi}
 We say that $h_n$ has an acceptable value if and only if the $(d_{ij}^n)=(d_{ij}^\infty+h_{ij}^n)$ satisfy Proposition~\ref{second_equality} iii).
 \end{defi}
    Assume $h_n$ has an acceptable value and a small norm; an easy computation proves that $h'_n\in{F}$ if and only if $\mathcal{T}_n$ is an isosceles tetrahedron.\\     
     We know that if $\mathcal{T}_0$ is not isosceles then, for all $n$, the tetrahedron $\mathcal{T}_n$ is not isosceles. Thus here $h'_n\notin{F}$ and if $h'_n=u_n+v_n$ is the decomposition associated to $\mathbb{R}^5=F\oplus{F\overset{\perp}{}}$, where $F\overset{\perp}{}$ is the orthogonal of $F$, then $(v_n)$ tends to $0$ and, for all $n$, $v_n\not=0$. Moreover ${h'_n}^T\Sigma{h'_n}={v_n}^T\Sigma{v_n}>0$ and ${h'_n}^TE{h'_n}={v_n}^TE{v_n}>0$. 
     
    $spectrum(E_{|F\overset{\perp}{}})=\{2d_{13}^\infty,8d_{12}^\infty,2d_{14}^\infty\}$ and the associated eigenvectors are $[-1,0,0,1,0]^T$, $[1,1,1,1,2]^T$ and $[0,-1,1,0,0]^T$.\\
  $spectrum(\Sigma_{|F\overset{\perp}{}})=\{2(4-{d_{12}^\infty})(4-d_{14}^\infty),8(4-d_{13}^\infty)(4-{d_{14}^\infty}),2(4-d_{12}^\infty)(4-d_{13}^\infty)\}$ and the associated eigenvectors are also $[-1,0,0,1,0]^T,[1,1,1,1,2]^T,[0,-1,1,0,0]^T$ ($E$ and $\Sigma$ are simultaneously diagonalizable).\\
  Let $m=\min\left \lbrace \dfrac{d_{13}^\infty}{(4-{d_{12}^\infty})(4-d_{14}^\infty)},\dfrac{d_{12}^\infty}{(4-{d_{13}^\infty})(4-d_{14}^\infty)},\dfrac{d_{14}^\infty}{(4-{d_{12}^\infty})(4-d_{13}^\infty)}\right \rbrace$,\\
$M=\max\left \lbrace \dfrac{d_{13}^\infty}{(4-{d_{12}^\infty})(4-d_{14}^\infty)},\dfrac{d_{12}^\infty}{(4-{d_{13}^\infty})(4-d_{14}^\infty)},\dfrac{d_{14}^\infty}{(4-{d_{12}^\infty})(4-d_{13}^\infty)}\right \rbrace$.

  Thus if $h_n$ has an acceptable value then $m\leq\dfrac{\epsilon'_n(h'_n)}{\sigma_n(h'_n)}\leq{M}$ and
$$\epsilon'_n=\epsilon_n+O(||h_n||^3)\geq{m}\sigma_n=-4m(4-d_{12}^\infty)(4-d_{13}^\infty)(4-d_{14}^\infty)\delta_n+O(||h_n||^3).$$

To finish the proof of $ii)$, it remains to show that $\{\epsilon_n$ and $\delta_n$ are $\Theta(||h_n||^2)\}$ $(*)$.\\ Indeed $(*)$ implies that for all sufficiently large $n$, $\epsilon_n>-k\delta_n$, choosing a positive number\\ $k<4\rho$ with $\rho=\min\{d_{13}^\infty(4-d_{13}^\infty),d_{12}^\infty(4-d_{12}^\infty),d_{14}^\infty(4-d_{14}^\infty)\}$.    
 \end{proof}
 \begin {rem}
 i) If the limit form is a regular tetrahedron then $m=M=\dfrac{3}{2}$.\\
 ii) It can be proved that ${\rho}\leq\dfrac{32}{9}$ with equality if and only if the limit form is regular.  
  \end {rem} 
 In the following we prove the property $(*)$.
 \subsection {Improvement of certain estimates}
 \begin{lem} If $(ijkl)\in\mathcal{U}=\{(1234),(1324),(1423)\}$ then
$d_{ij}^nd_{kl}^n-(d_{ij}^\infty)^2=O(||h_n||^2)$.
\end{lem}
\begin{proof}
Let $L_n=\dfrac{64}{(\sqrt{d_{12}^nd_{34}^n}+\sqrt{d_{13}^nd_{24}^n}+\sqrt{d_{14}^nd_{23}^n})^2}$. $(L_n)$ tends to $1$.\\
We have $(d_{ij}^\infty)^2=L_nd_{ij}^nd_{kl}^n$, $d_{ij}^nd_{kl}^n-(d_{ij}^\infty)^2=(\dfrac{1}{L_n}-1)(d_{ij}^\infty)^2$ and, consequently, $\dfrac{1}{L_n}-1=\dfrac{(\sqrt{d_{12}^nd_{34}^n}+\sqrt{d_{13}^nd_{24}^n}+\sqrt{d_{14}^nd_{23}^n})^2-8^2}{64}$. Thus, if we set $u_n=\sqrt{d_{12}^nd_{34}^n}+\sqrt{d_{13}^nd_{24}^n}+\sqrt{d_{14}^nd_{23}^n}-8$, we have to prove that $u_n=O(||h_n||^2)$:
  
\begin{tabular}{lll}
$u_n$ & $=$ & $\sum_{(ijkl)\in\mathcal{U}}\sqrt{(d_{ij}^\infty)^2+d_{ij}^\infty(h_{ij}^n+h_{kl}^n)+O(||h_n||^2)}-8 $ \\
 & $=$ &$\sum_{(ijkl)\in\mathcal{U}}d_{ij}^\infty\sqrt{1+\dfrac{h_{ij}^n+h_{kl}^n}{d_{ij}^\infty}+O(||h_n||^2)}-8$ \\
 & $=$ & $\sum_{(ijkl)\in\mathcal{U}}d_{ij}^\infty+\dfrac{h_{ij}^n+h_{kl}^n}{2}+O(||h_n||^2)-8$\\
 & $=$& $\dfrac{\sum_{i<j}h_{ij}}{2}+O(||h_n||^2)$\\
  & $=$& $=O(||h_n||^2)$, by proposition~\ref{comp_epsilon_delta}. \\
\end{tabular}  
  
\end{proof} 

\begin{lem}\label{hijO} If $(ijkl)\in\mathcal{U}$ then
$h_{ij}^n+h_{kl}^n=O(||h_n||^2)$.
\end{lem}
\begin{proof}
$d_{ij}^nd_{kl}^n-(d_{ij}^\infty)^2=d_{ij}^\infty(h_{ij}+h_{kl})+h_{ij}h_{kl}=O(||h_n||^2)$ by Lemma 1.
\end{proof}
\begin{lem}
For all sufficiently large $n$ there exists $(ijkl)\in\mathcal{U}$ such that\\ $|h_{ij}^n-h_{kl}^n|\geq{\dfrac{1}{\sqrt{6}}}||h_n||$.
\end{lem}
\begin{proof} By Lemma 2 there exists $A>0$ such that, for all $n$ and $(ijkl)\in\mathcal{U}$,\\ $|h_{ij}^n+h_{kl}^n|\leq{A||h_n||^2}$. Let $\epsilon\in(0,\dfrac{1}{6A})$. There exists $N$ such that $n\geq{N}\Rightarrow{||h_n||<\epsilon}$. Let $n\geq{N}$ be a fixed integer. We may assume $|h_{12}^n|=\sup_{i<j}{|h_{ij}^n|}$. Thus $|h_{12}^n|\geq\dfrac{||h_n||}{\sqrt{6}}\geq\dfrac{|h_{12}^n|}{\sqrt{6}}$ and $|h_{12}^n+h_{34}^n|\leq{6A(h_{12}^n)^2}$ and
  $|h_{12}^n-h_{34}^n|\geq{2}|h_{12}^n|-|h_{12}^n+h_{34}^n|\geq{2}|h_{12}^n|-6A(h_{12}^n)^2$.\\
 $0<|h_{12}^n|\leq||h_n||<\epsilon<\dfrac{1}{6A}$ then $6A(h_{12}^n)^2<|h_{12}^n|$ and $|h_{12}^n-h_{34}^n|\geq{|h_{12}^n|}\geq{\dfrac{1}{\sqrt{6}}}||h_n||$.
\end{proof}
\begin{prop}\label{est_delta}
If $(h_n)$ has acceptable values, then for all sufficiently large $n$, $\epsilon_n\geq{\lambda||h_n||^2}$ with $\lambda=\dfrac{\inf_{i<j}d_{ij}^\infty}{6}$; moreover $-\delta_n=\Theta(||h_n||^2)$.
\end{prop}
\begin{proof}
$\epsilon_n=(h_{12}^n-h_{34}^n)^2d_{12}^\infty+(h_{13}^n-h_{24}^n)^2d_{13}^\infty+(h_{14}^n-h_{23}^n)^2d_{14}^\infty$ and Lemma 3 give the first part.\\
For the second part, if $h_n$ has an acceptable value then we know, by proposition~\ref{comp_epsilon_delta}, that: \\
$i)$ $\epsilon'_n=\epsilon_n+O(||h_n||^3)$. Thus, by the first part, if $\mu<\lambda$ then for all sufficiently large $n$, $\epsilon'_n\geq{\mu}||h_n||^2$. \\
$ii)$ For all $n$, $\epsilon'_n(h'_n)\leq{M}\sigma_n(h'_n)$. Then $\sigma_n\geq{\dfrac{\mu}{M}}||h_n||^2$ (we can see $\sigma_n$ as a function of $h_n$) for all sufficiently large $n$.\\
$iii)$ $-\delta_n=\dfrac{\sigma_n}{\nu}+O(||h_n||^3)$ with $\nu=4(4-d_{12}^\infty)(4-d_{13}^\infty)(4-d_{14}^\infty)$.\\
Thus if $\mu_1<\lambda$, then for all sufficiently large $n$, $-\delta_n\geq{\dfrac{\mu_1}{M\nu}}||h_n||^2$.
\end{proof}
  \subsection {The main result in dimension three.}{\ }
  
  Let $r=\max\left \lbrace \dfrac{|d_{12}^\infty-2|}{2},\dfrac{|d_{13}^\infty-2|}{2},\dfrac{|d_{14}^\infty-2|}{2}\right \rbrace \in \left[\frac{1}{3},1\right)$. 
  \begin{thm}\label{convergence_planar_tetrahedron}
The sequences $(\mathcal{T}_{2i})_{i\in\mathbb{N}}$ and $(\mathcal{T}_{2i+1})_{i\in\mathbb{N}}$ are well defined and converge to two non planar isosceles tetrahedra that are symmetric with respect to $O$. Moreover the convergence is with at least geometric speed.
\end{thm}
\begin{proof}
${O{G_{n+1}}}^2=\dfrac{-\delta_n}{16}-\dfrac{1}{16^2}\epsilon_n+O(||h_n||^3)<\dfrac{-\delta_n}{16}+\dfrac{1}{16^2}k\delta_n+O(||h_n||^3)\sim{O{G_n}}^2\left(1-\dfrac{k}{16}\right)$ (according to the preceding proposition). Finally, for all sufficiently large $n$, $O{G_{n+1}}\leq{q}O{G_n}$ where $q<\sqrt{1-\dfrac{\rho}{4}}=r$. \\
\indent For all sufficiently large $n$, $G_n$ and $G_{n+1}$ are close to $O$ then close to the middle points of the segments $[A_nA_{n+1}]$ and $[A_{n+1}A_{n+2}]$; thus for all sufficiently large $n$, $A_nA_{n+2}\leq{3}G_nG_{n+1}$. Let $p\in\mathbb{N}^*$; $A_nA_{n+2p}
\leq{3}\sum_{k=n}^{n+2p-2}G_kG_{k+1}
\leq{3}\sum_{k=n}^{n+2p-2}(O{G_k}+O{G_{k+1}})\leq{6}\sum_{k=n}^{n+2p}O{G_k}$. The series $\sum{O{G_n}}$ converges, then $(A_{2n})_n$ is a Cauchy sequence; thus it converges to $A^\infty$. By the same way $(A_{2n+1})_n$ converges to $A'^\infty$, the symmetric of $A^\infty$ with respect to $O$. Moreover $A_{2n}A^\infty\leq{6}\sum_{k=2n}^\infty{O{G_k}}\leq\dfrac{6}{1-q}{O{G_{2n}}}$ for all sufficiently large $n$; therefore $A_{2n}A^\infty=O(q^{2n})$ and the sequences $(\mathcal{T}_{2i})$ and $(\mathcal{T}_{2i+1})$ converge with at least geometric speed.
\end{proof} 
\subsection {About the limit form} We consider the system of barycentric coordinates, for $\mathbb{R}^3$, determined by the vertices $A_0,B_0,C_0,D_0$ of $\mathcal{T}_0$. Let$(u,v,w,h)$ be four non all equal reals. $\{\alpha{A_0}+\beta{B_0}+\gamma{C_0}+\delta{D_0}\;;\; u\alpha+v\beta+w\gamma+h\delta=0\;,\;\alpha+\beta+\gamma+\delta=1\}$ is a plane $\Pi$. \\$(u,v,w,h)$ are said to be the barycentric coordinates of $\Pi$.\\
\indent We use the notations of $2.1.$ $\mathcal{T}_0$ is not isosceles; then $\{a'bc,b'ca,c'ab,a'b'c'\}$ are valid coordinates and define a plane $\Pi$ (Lemoine's plane) that does not intersect $\mathcal{S}$. Let $I_1,I_2$ be the two points satisfying the following conditions: they are symmetric with respect to $\Pi$ and inverse with respect to $\mathcal{S}$. They are said to be the isodynamic points of $\mathcal{T}_0$. There exist only two inversions leaving invariant $\mathcal{S}$ and mapping $A_0,B_0,C_0,D_0$ on the vertices of an isosceles tetrahedron; $I_1,I_2$ are the centers of these inversions (see~\cite{4} p. 184-186). The link with our paper is that the associated isosceles tetrahedra (that are not symmetric with respect to $O$) are  isometric to the limits of the sequence $(\mathcal{T}_i)$ but these couples are distinct.   
    
 \begin{rem} $\bullet$ The limit form is a regular tetrahedron if and only if the parameters of the non isosceles tetrahedron $\mathcal{T}_0$ satisfy $d_{12}d_{34}=d_{13}d_{24}=d_{14}d_{23}<\frac{64}{9}$ that is $\mathcal{T}_0$ is an isodynamic  tetrahedron.\\ 
$\bullet$ According to numerical experiments, we conjecture that $O{G_{n+1}}\sim{r}\times{O{G_n}}$\\
  (convergence with order one); in particular it can be observed the following facts:
\begin{enumerate}[i)]
 \item If the limit isosceles tetrahedron is almost flat, that is close to a rectangle, then the factor $r$ is close to $1$.
 \item  If the limit form is a regular tetrahedron then $O{G_{n+1}}\sim{\dfrac{1}{3}}\times{O{G_n}}$.
\end{enumerate}

\end{rem}
\section {Sequence of cyclic quadrilaterals}
\subsection {Degenerate simplices} Now we work in the Euclidean plane and we have a look on the case where $\mathcal{T}_0$ is a cyclic quadrilateral. We see the following as a degenerate case of the preceding study.
\begin{thm}
We consider a convex cyclic quadrilateral $\mathcal{T}_0=(A_0B_0C_0D_0)$ with circumcircle $\mathcal{C}$, the circle of center $O$ and radius $1$, and such that its vertices are pairwise distinct; we use the preceding iteration and produce a sequence of convex quadrilaterals $(\mathcal{T}_i)_i$. The sequences $(\mathcal{T}_{2i})_i$ and $(\mathcal{T}_{2i+1})_i$ are well defined and converge to rectangles that have same image, whose centroid is $O$ and whose lengths of the edges are $d_{13}^\infty=4$, $d_{12}^\infty=4\dfrac{\sqrt{d_{12}d_{34}}}{\sqrt{d_{13}d_{24}}}$ and $d_{14}^\infty=4-d_{12}^\infty$. If the limit form is not a square then $OG_{n+1}\sim\dfrac{|d_{12}^\infty-2|}{2}OG_n$ and the sequence converges with at least geometric speed; else the convergence is even faster. We study examples where the limit form is a square and the convergence is with order three.
\end{thm}
\indent We keep the preceding notations. In particular the parameters $(d_{ij}^n)_{i<j}$ refer to the square of the lengths of the edges (diagonals included) of $\mathcal{T}_n$.
\begin{proof}
$\bullet$ The beginning of the proof is similar to that of Theorem~\ref{convergence_ogn}. We use Theorem~\ref{inequalities} (note that inequality $iii)$ is useless because $Pt(\mathcal{T}_n)=0$); with $ii)$ we prove that the sequences $(d_{ij}^nd_{kl}^n)_n$ are increasing and if $i<j$ then the $(d_{ij}^n)_n$ have a positive lower bound; thus a cluster point has pairwise distinct vertices. With $i)$ we prove that a cluster point $\mathcal{T}$ is a rectangle. Assertion $ii)$ of Theorem~\ref{convergence_ogn} is always valid and gives explicitly the $(d_{ij}^\infty)_{i<j}$, the parameters of $\mathcal{T}$.\\
\indent The quadrilaterals are convex, then $d_{13}^\infty=4$ and $d_{14}^\infty=4-d_{12}^\infty$. Let $\mu=d_{12}^\infty$.\\
$\bullet$ $Relations\;between\,the\;h_{ij}$.\\
$Part\; I$.  There exist three algebraically independent relations between the $6$ parameters $(d_{ij}^n)_{i<j}$. The computation of the terms of degree $1$ of the Taylor series of the relation $\Gamma(A_n,B_n,C_n,D_n)=0$ gives $h_{12}^n+h_{14}^n+h_{23}^n+h_{34}^n-h_{24}^n-h_{13}^n=O(||h_n||^2)$.\\
\indent Moreover the triangles $(A_nB_nC_n)$, $(B_n,C_n,D_n)$ admit a circumscribed circle of radius $1$; for example for the first triangle the relation is ${a_n}^2{b_n}^2{c_n}^2=16\times{S_n}^2$ (the value of the area $S_n$ of the triangle $(A_nB_nC_n)$ is given in Proposition~\ref{nsc_tetrahedron} i)). 
The computation of the terms of degree $1$ of the Taylor series of these relations gives $h_{13}^n=O(||h_n||^2)$ and $h_{24}^n=O(||h_n||^2))$, thus $\delta_n=O(||h_n||^2)$. 
\begin{defi}
We say that $h_n$ has an acceptable value if and only if the $(d_{ij}^n)=(d_{ij}^\infty+h_{ij}^n)$ satisfy the preceding three relations.
 \end{defi}
\indent Using the proof of Lemma~\ref{hijO}, we show that $h_{12}^n+h_{34}^n=O(||h_n||^2)$ and $h_{14}^n+h_{23}^n=O(||h_n||^2)$.\\
Thus ${h_{12}^n}^2+{h_{14}^n}^2=\Theta(||h_n||^2)$.\\
$Part \;II.$ The computation of the terms of degree at most two of the Taylor series of the relations associated to the triangles $(A_nB_nC_n)$, $(B_nC_nD_n)$ give: $h_{13}^n\mu(4-\mu)+(h_{12}^n+h_{23}^n)^2=O(||h_n||^3)$ and $h_{24}^n\mu(4-\mu)+(h_{23}^n+h_{34}^n)^2=O(||h_n||^3)$ that is $h_{13}^n=\dfrac{-(h_{12}^n-h_{14}^n)^2}{\mu(4-\mu)}+O(||h_n||^3)$ and $h_{24}^n=\dfrac{-(h_{12}^n+h_{14}^n)^2}{\mu(4-\mu)}+O(||h_n||^3)$.\\
\indent The relation $\Gamma(A_n,B_n,C_n,D_n)=0$ gives by a similar way\\ $\mu(\mu-4)(h_{12}^n+h_{14}^n+h_{23}^n+h_{34}^n-h_{24}^n-h_{13}^n)+\mu{h_{14}^n}^2+(4-\mu){h_{12}^n}^2=O(||h_n||^3)$.\\
Using the last three relations we can deduce: $\delta_n=\dfrac{-{h_{14}^n}^2}{\mu}-\dfrac{{h_{12}}^2}{4-\mu}+O(||h_n||^3)$  and\\
 $\delta_n=\Theta(||h_n||^2)$. Moreover $\epsilon_n=4\mu{h_{12}^n}^2+4(4-\mu){h_{14}^n}^2+O(||h_n||^3)$.\\
Finally ${OG_{n+1}}^2=\dfrac{(\mu-2)^2}{64}\left(\dfrac{{h_{14}^n}^2}{\mu}+\dfrac{{h_{12}}^2}{4-\mu}\right)+O(||h_n||^3)=\dfrac{(\mu-2)^2}{4}{OG_n}^2+O(||h_n||^3)$.\\
$Case\; 1$: $\mu\not=2$; then the limit form is not a square and $OG_{n+1}\sim\dfrac{|\mu-2|}{2}OG_n$.
 The convergence is with order one. \\
 $Case\; 2$: $\mu=2$; then the limit form is a square and $OG_{n+1}=O(||h_n||^{\frac{3}{2}})$ or $OG_{n+1}=O({OG_n}^{\frac{3}{2}})$. In fact the convergence is much faster than in the preceding case. For example if the polar angles of $A_0,B_0,C_0,D_0$ are $\{0,\arccos(0.923827833284),\arccos(-0.8),-\arccos(0.9)\}$ then we obtain a quasi square after few iterations.\\
 \indent Now we can conclude: in both cases the series $\sum_nOG_n$ converges; thus we may reason as in Theorem~\ref{convergence_planar_tetrahedron} and we obtain that the sequences $(\mathcal{T}_{2i})_{i\in\mathbb{N}}$ and $(\mathcal{T}_{2i+1})_{i\in\mathbb{N}}$ converge to rectangles that have same image with at least geometric speed.
\end{proof}
\subsection {A particular case} 
 The nullspace $N$ of the symmetric matrix $\Sigma$ is the hyperplane defined by $h_{13}^n=h_{24}^n$. Assume $h_n$ has an acceptable value and a small norm; then a geometric argument or an algebraic computation shows that $h'_n\in{N}$ is equivalent to $\mathcal{T}_n$ is an isosceles trapezoid.
\begin{prop} Preserving the assumptions of theorem 4, we assume there exists $k$ such that $\mathcal{T}_k$ is an isosceles trapezoid but not a rectangle. If moreover $d_{12}^\infty=2$ then $OG_{n+1}\sim{OG_n}^3$.
\end{prop}
\begin{proof}
For all $n\geq{k}$, $\mathcal{T}_n$ is an isosceles trapezoid that admits $OG_k$ as a line of symmetry. Therefore $\mathcal{T}_n$ cannot turn around $O$; we know explicitly the limit form of $\mathcal{T}_n$, then $\mathcal{T}_n$ converges to a rectangle that admits $OG_k$ as a line of symmetry.\\
\indent We study the rate of convergence using an explicit calculation. We may assume that the line of symmetry is the axis of abscissas. The vertical edges $A_nD_n$ and $B_nC_n$ have $a_n,b_n$ as abscissas; then the abscissa of the centroid of $\mathcal{T}_n$ is $g_n=\dfrac{a_n+b_n}{2}$. An easy calculation gives:\\ $a_{n+1}=-\dfrac{{a_n}^2b_n+2a_n{b_n}^2+{b_n}^3-4a_n}{{a_n}^2-2a_nb_n-3{b_n}^2+4},\\ b_{n+1}=-\dfrac{{b_n}^2a_n+2b_n{a_n}^2+{a_n}^3-4b_n}{{b_n}^2-2b_na_n-3{a_n}^2+4}$.\\
 We know that $a_n+b_n$ tends to $0$ and ${a_n}^2$ tends to the known expression $\dfrac{d_{12}^\infty}{4}$.\\ $g_{n+1}\sim\dfrac{(a_n+b_n)(a_n^4-4{a_n}^3b_n-10{a_n}^2{b_n}^2+16{a_n}^2-4a_n{b_n}^3+16{b_n}^2-16+{b_n}^4)}{-32}$.\\ 
\indent $i)$ $d_{12}^\infty\not=2$. The limit form is not a square and we obtain $g_{n+1}\sim{\left(1-\dfrac{d_{12}^\infty}{2}\right)g_n}$ (convergence with order one). If $d_{12}^\infty>2$ then, for all sufficiently large $n$, \mbox{$O\in]G_nG_{n+1}[$.}\\
\indent $ii)$ $d_{12}^\infty=2$. The limit form is a square that is the parameters of $\mathcal{T}_0$ satisfy the relation $(d_{12})^2=d_{14}d_{23}$. Let $a_n=\dfrac{1}{\sqrt{2}}+u_n,b_n=-\dfrac{1}{\sqrt{2}}+v_n$ where $u_n$ and $v_n$ tend to $0$ and \mbox{$(u_n,v_n)\not=(0,0)$.}
 We calculate the Taylor series of the preceding relation and we consider the terms of degree at most $2$; we obtain $16\sqrt{2}(u_n-v_n)=-8(3u_n^2+3v_n^2+2u_nv_n)+O(||(u_n,v_n)||^3)$. Therefore $u_n-v_n\sim\dfrac{-1}{2\sqrt{2}}(3u_n^2+3v_n^2+2u_nv_n)$. $u_n-v_n=O(||(u_n,v_n)||^2)$ thus \mbox{$u_n+v_n=\Theta(||(u_n,v_n)||)$. }
We deduce easily the estimate $g_{n+1}\sim\dfrac{(a+b)(-4(u_n+v_n)^2+O(||(u_n,v_n)||^3)}{-32}\sim\dfrac{-4(u_n+v_n)^3}{-32}$ \\
or $g_{n+1}\sim{g_n}^3$ (convergence with order three) what is astonishing.\\
 \indent For example if $a_0=0.955,b_0=0.12237784429$, then we obtain a quasi square after three iterations.
 \end{proof}
 \subsection {About the limit form} As in the case of a tetrahedron there exist two inversions leaving $\mathcal{C}$ invariant and mapping $A_0,B_0,C_0,D_0$ on the vertices of a rectangle; the centers of these inversions are inverse with respect to $\mathcal{C}$. The associated rectangles are isometric to the limit of the sequence $(\mathcal{T}_i)$. In particular the limit form  is a square if and only if the parameters of the non rectangular quadrilateral $\mathcal{T}_0$ satisfy $d_{12}d_{34}=d_{14}d_{23}<4$ that is $\mathcal{T}_0$ is a harmonic (or isodynamic) quadrilateral. 
 \begin{rem} $i)$  If the limit form is a square then the convergence seems to be with order three; more precisely we conjecture that $OG_{n+1}\sim{OG_n}^3$.\\  
$ii)$ It can be observed a strange phenomenon: if $d=3$ and the limit form is close to a flat tetrahedron then we have seen that $r\approx{1}$. If $d=2$ then $r$ is in general far from $1$; in particular if the limit form is close to a square then $r$ is close to $0$.\\
\indent An explanation is that, if $d=3$, then $G_n$ tends to $O$ in such a way that the angles between $\overrightarrow{OG_n}$ and the edges of $\mathcal{T}_n$ are close to $\dfrac{\pi}{2}$; thus if $d=2$, then we study a quasi orthogonal projection of the preceding $G_n$ that converges faster than this $G_n$.
\end{rem}
 
    \section{Solution of the case $d=2$}
    \indent Now we consider a dynamical system of triangles.
    \subsection{Some remarks.} 
    
\begin{enumerate}[i)]
\item This problem has been solved in~\cite {9} by one of the authors. The proof in~\cite {9} is partially geometric; now we give a complete proof that is essentially algebraic. Moreover we give in Proposition~\ref{ordre_conv_triangle} a much better estimate of $OG_n$.
\item $\phi$ is not one to one: indeed, if $\mathcal{T}_1$ is a generic triangle, then there exists two triangles $\mathcal{T}_0$ such that $\phi(\mathcal{T}_0)=\mathcal{T}_1$.
\end{enumerate}   

    \subsection{The parameters}
    Let $a=B_0C_0,b=B_0C_0,c=C_0A_0$; the parameters of $\mathcal{T}_0$ are: $s=a^2+b^2+c^2,t=a^2b^2+b^2c^2+c^2a^2,u=a^2b^2c^2$. \\
   Let us recall that the circumradius of $\mathcal{T}_0$ is $1$; then $u=4t-s^2$ and $t>\dfrac{{s}^2}{4}$. Moreover $u=16\times{S^2}$ where $S$ is the area of $(A_0B_0C_0)$.\\
   ${OG_0}^2=1-\dfrac{s}{9}$ and $0<s\leq{9}$. Therefore $s=9\Leftrightarrow{G_0=O}\Leftrightarrow{(A_0B_0C_0)}$ is an equilateral triangle.\\
   $s^2-3t=a^4-a^2(b^2+c^2)+b^4+c^4-b^2c^2$; this is a polynomial in $a^2$ with discriminant $-3(b^2-c^2)^2\leq{0}$. Thus $t\leq\dfrac{s^2}{3}$ and $t=\dfrac{s^2}{3}\Leftrightarrow$ $(A_0B_0C_0)$ is equilateral. 
   \subsection{Inequalities}
   Here $A_0,B_0,C_0$ can be on a line but are not all equal; thus $s>0,t>0$. $s_1,t_1,u_1$ are the parameters of $\mathcal{T}_1$; we obtain: \\$s_1=\dfrac{s^2(6t-s^2)}{D},t_1=\dfrac{s^4t(9t-2s^2)}{D^2},u_1=4t_1-{s_1}^2$ where\\ $D=-4s^3+18st-108t+27s^2$.
   $D>0$ because the numerator of $s_1$ is positive; moreover $D$ is bounded.\\
  $\bullet$ $s_1-s=\dfrac{3s(9-s)(4t-s^2)}{D}\geq{0}$. Moreover $s_1=s$ if and only if $s=9$ or $u=0$ that is $A_0B_0C_0$ is equilateral or flat. \\
  $\bullet$ $t_1-t=\dfrac{9t(4t-s^2)(9(\dfrac{s^2}{3}-t)(s-6)^2+s^2(9-s)(s-3))}{D^2}$. If $s\geq{3}$ then $t_1\geq{t}$. \\
  $\bullet$ $\dfrac{u_1}{u}=\left(\dfrac{s^3}{D}\right)^2$. $u_1\geq{u}\Leftrightarrow{s^3\geq{D}}\Leftrightarrow{\nu=s^2(9-s)}+18\left(\dfrac{s^2}{3}-t\right)(s-6)\geq{0}$.\\
   If $s\geq{6}$ then $\nu\geq{0}$. Now we assume $s<6$; $t>\dfrac{s^2}{4}$ implies that 
   $\nu\geq\dfrac{s^3}{2}$. \\
   Therefore $u_1\geq{u}$.
   \subsection{Convergence of the triangles}
   We assume $\mathcal{T}_0$ is not an equilateral triangle. $s_n,t_n,u_n$ refer to the parameters of $\mathcal{T}_n$.
   \begin{prop}
   Let $\mathcal{T}$ be a cluster point of the bounded sequence $(\mathcal{T}_n)_n$; then  $\mathcal{T}$ is a non flat equilateral triangle. Moreover the lengths of the edges of $\mathcal{T}_n$ converge to  $\sqrt{3}$.
\end{prop}
\begin{proof}
$\bullet$ The sequence $(u_n)$ is increasing; thus it converges to $u^\infty>0$. Let $\mathcal{T}$ be a cluster point of the sequence $(\mathcal{T}_n)$. The sequence $(s_n)$ is increasing; thus it converges to $s^\infty>0$. With a proof similar to that used in the theorem 3 we show that: 
\begin{enumerate}
\item $s^\infty,u^\infty$ are parameters of $\mathcal{T}$; therefore $\mathcal{T}$ is not flat.
\item $\mathcal{T}$ is an equilateral triangle and $s^\infty=9,u^\infty=27$.
\end{enumerate}
$\bullet$ Moreover for all sufficiently large $n$, $s_n\geq{3}$ and the sequence $(t_n)$ is increasing then convergent; finally  $(t_n)$ converge to $27$, the corresponding parameter of $\mathcal{T}$.
The sequences $(s_n),(t_n),(u_n)$ converge to $9,27,27$; therefore $(a_n),(b_n),(c_n)$, the lengths of the edges of $\mathcal{T}_n$, converge to $\sqrt{3}$.
\end{proof}
Let ${a_n}^2=3+h_n,{b_n}^2=3+k_n,{c_n}^2=3+l_n,\delta_n=(h_n,k_n,l_n)$. $||.||$ refers to the euclidean norm. Now we show an important estimate.
\begin{lem}
$h_n+k_n+l_n\sim\dfrac{1}{3}(h_nk_n+k_nl_n+l_nh_n)\sim\dfrac{-1}{6}||\delta_n||^2$ when $n$ tends to $\infty$.
\begin{proof}
$u=t-4s^2\Leftrightarrow{3(h_n+k_n+l_n)}=-(h_n+k_n+l_n)^2+(h_nk_n+k_nl_n+l_nh_n)$\\
$-h_nk_nl_n$. Thus $3(h_n+k_n+l_n)={-(h_n+k_n+l_n)^2+(h_nk_n+k_nl_n+l_nh_n)}+O(||\delta_n||^3)={h_nk_n+k_nl_n+l_nh_n}+O(||\delta_n||^3)$.\\
$h_nk_n+k_nl_n+l_nh_n=\dfrac{1}{2}(h_n+k_n+l_n)^2-\dfrac{1}{2}(h_n^2+k_n^2+l_n^2)\sim{-\dfrac{1}{2}(h_n^2+k_n^2+l_n^2)}$. 
\end{proof}
\end{lem}
\begin{prop}\label{ordre_conv_triangle}
If $G_n$ is the centroid of $\mathcal{T}_n$ then $OG_{n+1}\sim{OG_n}^2$ when $n$ tends to $\infty$. Thus the sequence $(OG_n)$ converges with order $2$.
\end{prop}
\begin{proof}
 ${OG_n}^2=\dfrac{h_n+k_n+l_n}{-9}$; ${OG_{n+1}}^2=1-\dfrac{s_{n+1}}{9}$. Using Maple and Lemma 1, we obtain the Taylor series of ${OG_{n+1}}^2$ with the precision  $O(||{\delta_n}||^5)$:\\
${OG_{n+1}}^2=\dfrac{N_n}{-81^2}$ where\\
 $N_n=81(h_n+k_n+l_n)^2+18(h_n+k_n+l_n)({2h_n}^2+2{k_n}^2+2{l_n}^2+h_nk_n+k_nl_n+l_nh_n)+O(||{\delta_n}||^5)\\ \phantom{N_n} = 27(h_n+k_n+l_n)(3(h_n+k_n+l_n)-2(h_nk_n+k_nl_n+l_nh_n))+O(||{\delta_n}||^5)$\\
$\sim{-81(h_n+k_n+l_n)^2}$.
 Finally ${OG_{n+1}}^2\sim\dfrac{(h_n+k_n+l_n)^2}{81}={OG_n}^4$.
\end{proof}
\begin{rem}
$\bullet$ We can deduce that there exists $\lambda\in(0,1)$, that depends upon $a,b,c$, such that $OG_n\sim\lambda^{2^n}$. If $\mathcal{T}_0$ is close to a flat triangle then $\lambda$ is close to $1$. Of course if $\mathcal{T}_0$ is close to an equilateral triangle then $\lambda$ is close to $0$ (see $\mathcal{T}_0$ as the result of a large number of iterations.)\\
$\bullet$ The result obtained in~\cite{9}: $OG_{n+1}=O({OG_n}^2)$ is weaker and does not give the preceding estimate of $OG_n$.
\end{rem}

\subsection {The main result in dimension two}
\begin{thm}The sequences of triangles $(\mathcal{T}_{2i})_{i\in\mathbb{N}}$ and $(\mathcal{T}_{2i+1})_{i\in\mathbb{N}}$ are well defined and converge with at least quadratic speed to two equilateral triangles that are symmetric with respect to $O$.
\end{thm}  

\begin{proof}
For all sufficiently large $n$, $G_n$ and $G_{n+1}$ are close to $O$ then close to the middle points of the segments $[A_nA_{n+1}]$ and $[A_{n+1}A_{n+2}]$; thus for all sufficiently large $n$, $A_nA_{n+2}\leq{3}G_nG_{n+1}$. Let $p\in\mathbb{N}^*$; $A_nA_{n+2p}
\leq{3}\sum_{k=n}^{n+2p-2}G_kG_{k+1}\leq{3}\sum_{k=n}^{n+2p-2}(OG_k+OG_{k+1})\leq{6}\sum_{k=n}^{n+2p}OG_k$. The series $\sum{OG_n}$ converges, then $(A_{2n})_n$ is a Cauchy sequence; thus it converges to $A^\infty$. By the same way $(A_{2n+1})_n$ converges to $A'^\infty$, the symmetric of $A^\infty$ with respect to $O$.

Moreover $A_{2n}A^\infty\leq{6}\sum_{k=2n}^\infty{OG_k}\leq{12}\times{OG_{2n}}$ for all sufficiently large $n$; therefore $A_{2n}A^\infty=O(\lambda^{2^{2n}})$ and the sequences $(\mathcal{T}_{2i})$ and $(\mathcal{T}_{2i+1})$ converge with at least quadratic speed.
\end{proof} 
    \section{Conclusion}
    We mention some questions we can ask about these dynamical systems. 
\begin{enumerate}[i)]
 \item  Can one generalize our results concerning the $d$-simplices for $d>3$? We note that the complexity of the computations increases very quickly.
 \item   What occurs if $\mathcal{T}_0$ is a degenerate $d$-simplex that admits a circumsphere in $\mathbb{R}^e$ with $e<d$?  For example, $\mathcal{T}_0$ could be a convex polyhedron that admits a circumsphere in $\mathbb{R}^3$. 
 \item  More generally what occurs if we replace the centroid of $\mathcal{T}_i$ with some barycenter of the vertices of $\mathcal{T}_i$?  
\end{enumerate}

\bibliographystyle{plain}
\bibliography{ArxivTetralte2}
%
%
%
%
%
%
%
%
%
%
%
%

\end{document}